\documentclass[letterpaper,english]{amsart}
\usepackage{amsmath,amsthm,amssymb}
\usepackage{mathrsfs}
\usepackage[all,cmtip]{xy}
\usepackage{tikz-cd}
\usepackage{cite}
\usepackage{hyperref}
\usepackage{graphicx}
\usepackage{color}
\usepackage{colonequals} 
\usepackage[shortlabels]{enumitem}
\usepackage{algorithm}
\usepackage{algorithmic}

\theoremstyle{plain}
\newtheorem{lem}{Lemma}
\newtheorem{lemma}[lem]{Lemma}

\newtheorem{theorem}[lem]{Theorem}

\newtheorem{proposition}[lem]{Proposition}

\newtheorem{fact}[lem]{Fact}

\newtheorem{assumption}[lem]{Assumption}

\theoremstyle{definition}

\newtheorem{definition}[lem]{Definition}
\newtheorem{example}[lem]{Example}
\newtheorem{remark}[lem]{Remark}

\newtheorem{notation}[lem]{Notation}

\numberwithin{equation}{section}
\numberwithin{lem}{section}

\newcommand{\mathfont}{\mathbf}

\newcommand{\Z}{\mathfont Z}

\newcommand{\Q}{\mathfont Q}
\newcommand{\QQ}{\mathfont Q}

\newcommand{\F}{\mathfont F}
\newcommand{\FF}{\mathfont F}

\newcommand{\fp}{\mathfrak{p}}

\newcommand{\cD}{\mathcal{D}}

\newcommand{\cO}{\mathcal{O}}

\DeclareFontFamily{OT1}{rsfs}{}
\DeclareFontShape{OT1}{rsfs}{n}{it}{<-> rsfs10}{}
\DeclareMathAlphabet{\mathscr}{OT1}{rsfs}{n}{it}

\newcommand \tensor[1] {\otimes_{#1}}

\DeclareMathOperator{\ord}{ord}

\DeclareMathOperator{\rank}{rank}

\DeclareMathOperator{\Mat}{Mat}

\newcommand{\Zp}{\mathfont{Z}_p}
\newcommand{\Zq}{\mathfont{Z}_q}

\newcommand{\Qq}{\mathfont{Q}_{q}}

\newcommand{\Fp}{\mathfont{F}_p}

\newcommand{\PP}{\mathfont{P}}
\DeclareMathOperator{\Spec}{Spec}

\DeclareMathOperator{\Jac}{Jac}

\DeclareMathOperator{\GL}{GL}

\newcommand{\Xlift}{\widetilde{X}}
\newcommand{\Xgen}{\widetilde{X}_{\Qq}}

\newcommand{\Vlift}{\widetilde{V}}
\newcommand{\Ulift}{\widetilde{U}}
\newcommand{\Ugen}{\widetilde{U}_{\Qq}}
\newcommand{\Qlift}{\widetilde{Q}}
\newcommand{\red}{\textrm{red}}

\newcommand{\pr}{\operatorname{pr}}

\DeclareMathOperator{\HH}{\mathrm{H}}
\newcommand{\HdR}{\HH_{\textrm{dR}}}
\newcommand{\Hrig}{\HH_{\textrm{rig}}}
\newcommand{\Hcrys}{\HH_{\textrm{crys}}}
\newcommand{\rig}{\textrm{rig}}
\newcommand{\crys}{\textrm{crys}}

\newcommand{\Rdagger}{\mathcal{R}^\dagger}

\newcommand{\powerseries}[1]{[\![#1]\!]}
\newcommand{\Supp}{\operatorname{Supp}}
\newcommand{\Image}{\operatorname{Im}}
\newcommand{\Bcrys}{\mathcal{B}_{\textrm{crys}}}
\newcommand{\Lcrys}{\Lambda_{\textrm{crys}}}
\newcommand{\Brig}{\mathcal{B}_{\textrm{rig}}}
\newcommand{\Lrig}{\Lambda_{\textrm{rig}}}
\newcommand{\Nint}{N_{\textrm{int}}}

\newcommand{\Bone}{\mathcal{B}_1}
\newcommand{\Btwo}{\mathcal{B}_2}
\newcommand{\omegazero}{\widetilde{\omega}}

\title[Computing Crystalline Cohomology]{Computing Crystalline Cohomology and $p$-Divisible Groups for Curves over Finite Fields} %
\author{Jeremy Booher}
\address{University of Florida}
\email{jeremybooher@ufl.edu}

\begin{document} 

\begin{abstract} 
Let $X$ be a smooth projective curve over a finite field of characteristic $p$.  We describe and implement a practical algorithm for computing the $p$-divisible group $\Jac(X)[p^\infty]$ via computing its Dieudonn\'{e} module, or equivalently computing the Frobenius and Verschiebung operators on the first crystalline cohomology of $X$.  We build on Tuitman's $p$-adic point counting algorithm, which computes the rigid cohomology of $X$ and requires a ``nice'' lift of $X$ to be provided.  
\end{abstract}

\maketitle

\section{Introduction}

Let $X$ be a smooth projective curve of genus $g$ over a finite field $\F_q$ with $q = p^r$.  The $p$-divisible group $\Jac(X)[p^\infty]$ is built out of the finite flat group schemes $\Jac(X)[p^n]$.  It determines the zeta function, Newton polygon, and Ekedahl-Oort type of $X$, and is a geometric generalization of the $p$-part of the class group of $\F_q(X)$.  Indeed, the $p$-primary part of the class group is the $\F_q$-points of the (not necessarily \'{e}tale) $\Jac(X)[p^\infty]$.  

Equivalently, one can consider the first crystalline cohomology $\Hcrys^1(X)$.  Set $\Zq = W(\F_q)$, the unique unramified extension of $\Zp$ with residue field $\F_q$, and let $\Qq$ be the field of fractions.  Then $\Hcrys^1(X)$ is a free  $\Zq$-module of rank $2g$ and is equipped with Frobenius and Verschiebung endomorphisms $F$ and $V$.  It is naturally isomorphic to the Dieudonn\'{e} module of $\Jac(X)[p^\infty]$, so determining $\Hcrys^1(X)$ with the actions of $F$ and $V$ is equivalent to determining the $p$-divisible group.  (See Section~\ref{sec:cohomology} for a review of the cohomology theories and $p$-divisible groups.)

We introduce an efficient and practical algorithm for computing $\Hcrys^1(X)$.  More precisely, given an integer $N$ we compute the Dieudonn\'{e} module of $\Jac(X)[p^N]$, or equivalently matrices for $F$ and $V$ on $\Hcrys^1(X)$ with $p$-adic precision $N$.  We build on $p$-adic point counting algorithms which work by approximating the Frobenius morphism on the rigid cohomology of $X$, in particular work of Tuitman \cite{TuitmanI,TuitmanII}.  Tuitman works with a plane curve model $Q(x,y)=0$ for $X$, and exploits the projection maps to the coordinate axes. 

Our main contribution is to realize the crystalline cohomology as the de Rham cohomology of a lift $\Xlift/\Zq$ of $X$ and determining how this sits as a lattice inside the rigid cohomology $\Hrig^1(X)$. An idealized version is presented as Algorithm~\ref{alg1}.

\begin{algorithm}
\caption{Idealized Algorithm for Computing Crystalline Cohomology} 
 \label{alg1}
\begin{algorithmic}[1]
    \STATE Compute a $\Zq$-basis $\Bcrys$ for $\HdR^1(\Xlift) \cong \Hcrys^1(X)$ (see Fact~\ref{fact:differentialssecondkind}).
    \STATE Compute a $\Qq$-basis $\Brig$ for $\Hrig^1(X)$ (see Remark~\ref{remark:Espaces}).
    \STATE Compute the Frobenius on $\Hrig^1(X)$ using Tuitman's algorithm, and express it as a matrix relative to $\Brig$.
   \STATE  Represent $\Hcrys^1(X)$ as a $\Zq$-lattice in the $\Qq$-vector space $\Hrig^1(X)$ (see Fact~\ref{fact:rigid}), and hence obtain a matrix for Frobenius on $\Hcrys^1(X)$ relative to $\Bcrys$.
   \STATE Recover $V$ from the relation $F V = p$.
\end{algorithmic}
\end{algorithm}

The two key ingredients are an algorithm to compute $\HdR^1(\Xlift/\Zq)$ (see Section~\ref{sec:computing cohomology}) and understanding how much precision is needed internally in Tuitman's algorithm to guarantee the Dieudonn\'{e} module is correct modulo $p^N$ (Section~\ref{sec:precision}).  The former is non-trivial as $\Zq$ is not a field.  The latter encounters errors in the implementation of Tuitman's algorithm \cite{pcc} which we also address.  

We analyze our algorithm in Section~\ref{sec:analysis}.  We track several necessary assumptions, some inherited from Tuitman's work, which guarantee the algorithm works and the running time is polynomial in $p$, $d_x$, $d_y$, $r$, and $N$.  Here $d_x$ and $d_y$ are the degrees of the projection maps to the coordinate axes.

\begin{remark}
A key feature of our situation is that we consider $p$-power torsion in $\Jac(X)$ in characteristic $p$.  If $\ell \neq p$, there is considerable work studying the $\ell$-power torsion in $\Jac(X)$ and $\ell$-adic cohomology of $X$, for example in the context of point counting \cite{pila,couveignes,hi98,jin20,levrat24}.  Previous algorithmic work on $p$-power torsion in characteristic $p$ has only been able to access the isogeny class of $\Jac(X)[p^\infty]$ (via rigid cohomology, which has $\Qq$-coefficients) as used in $p$-adic point counting algorithms, or the $p$-torsion (via the de Rham cohomology of $X$  \cite{weir_code}, which has $\F_q$-coefficients).  
\end{remark}

\begin{remark}
We implement our algorithm for curves over $\Fp$ \cite{github}.  We use Magma \cite{magma}, and build on Tuitman's implementation \cite{pcc} of his algorithm.  We also incorporate a few tweaks from Balakrishna and Tuitman's work on Coleman integration \cite{BT20}.
As in Tuitman's implementation, we attempt to compute some auxiliary information in an ad hoc way which usually works but is not guaranteed.  See Assumption~\ref{assumption:nicelift} and Assumption~\ref{assumption:funcdif}.  The algorithm is practical: see Example~\ref{ex:precision} and Remark~\ref{remark:runningtime} as well as additional examples in \cite{github}.
\end{remark}

\begin{remark}
We also explain a precision issue in Tuitman's implementation \cite{pcc} (present in the version incorporated into Magma), a sporadic source of errors computing zeta functions.  See Section~\ref{ss:precision example}, and our github repository \cite{github2} for a version of the point counting algorithm which works around it.
\end{remark}

\subsection{Acknowledgments}  We thank Jen Balakrishna, Bryden Cais, Richard Crew, Kiran Kedlaya, and Jan Tuitman for helpful conversations.  

\section{The Situation} \label{sec:situation}

\subsection{Theoretical Setup} As in the introduction, let $p$ be a prime and $q = p^r$ a prime power.  Let $\Z_q$ be the unramified extension of $\Z_p$ with residue field $k=\FF_q$, and $\QQ_q$ be the fraction field of $\Z_q$.

Let $X$ be a smooth projective curve over $k$, with a finite separable map $x : X \to \PP^1_k$ of degree $d_x$.  
We inherit the setup from \cite{TuitmanII}.  In particular, we represent $X$ via a (possibly singular) plane curve $Q(x,y)=0$ for some $Q \in \F_q[x,y]$ and the morphism $x$ is projection onto the $x$-coordinate. Note $d_x$ is the degree of $Q$ in $x$.  The degree of the projection onto the $y$-coordinate $d_y$ is likewise the degree of $Q$ in $y$. We furthermore assume $Q$ is monic in $y$.

Next we lift to characteristic zero, requiring a $\Qlift \in \Z_q[x,y]$ that is monic in $y$, has the same monomials as $Q$, and reduces to $Q$.  We use the polynomial $\Qlift$ to represent a lift of $X$ to characteristic zero, which will be quite nice if we make some additional assumptions.  Letting $\Delta \in \Z_q[x]$ denote the discriminant of $\Qlift$ with respect to $y$ and $r=\Delta/\gcd(\Delta, \Delta')$, the natural map
\begin{equation}
x: \Ulift \colonequals \Spec \Z_q[x,y,1/r]/(\Qlift) \to \Vlift \colonequals \Spec \Z_q[x,1/r]
\end{equation}
is a finite \'{e}tale morphism.

We impose essentially the same assumptions as in Tuitman's work, taking into account the erratum \cite{TuitmanErrata}.

\begin{assumption} \label{assumption:nicelift}
We assume that:
\begin{enumerate}
\item we are given matrices $W^0 \in \GL_{d_x}(\Zq[x,1/r])$ and $W^\infty \in \GL_{d_x}(\Zq[x,1/x,r])$ such that letting
\[
b^0_j \colonequals \sum_{i=0}^{d_x-1} W^0_{i+1,j+1} y^i, \quad \text{and} \quad b^\infty_j \colonequals \sum_{i=0}^{d_x-1} W^\infty_{i+1,j+1} y^i 
\]
for $0\leq j < d_x$ the elements $b^0_0,\ldots, b^0_{d_x-1}$ form an integral basis for the function field $\Qq(\Ulift)$ over $\Qq[x]$ and (after reducing modulo $p$) for the function field $\F_q(U)= \F_q(X)$ over $\F_q[x]$.  We assume the same for $b^\infty_0,\ldots,b^\infty_{d_x-1}$ over $\Qq[1/x]$ and (after reducing modulo $p$) over $\F_q[1/x]$.

\item  the discriminant of $r$ is in $\Zq^\times$.

\item the discriminants of the $\Zq$-algebras $(\mathcal{R}^0/(r(X)))_{\red}$ and $(\mathcal{R}^\infty/(1/x))_{\red}$ are in $\Zq^\times$, where $\mathcal{R}^0$ is the $\Zq[x]$-module generated by $b^0_0,\ldots,b^0_{d_x-1}$ and $\mathcal{R}^\infty$ is the $\Zq[1/x]$-module generated by $b^\infty_0,\ldots,b^\infty_{d_x-1}$.  (Note for a ring $R$, the reduced ring $R_\red$ is the quotient of $R$ by its nilradical.)

\end{enumerate}
\end {assumption}

Geometrically, this assumption means that $x : \Ulift \to \Vlift$ admits a good compactification \cite[Proposition 2.3]{TuitmanII}.  More precisely: 
\begin{itemize}
\item  there is a smooth relative divisor $\mathcal{D}_{\PP^1_{\Zq}}$ such that $\Vlift = \PP^1_{\Zq} \backslash \mathcal{D}_{\PP^1_{\Zq}}$, and
\item there is a smooth proper curve $\Xlift$ over $\Zq$ and a smooth relative divisor $\mathcal{D}_{\Xlift}$ on $\Xlift$ such that $\Ulift = \Xlift \backslash \mathcal{D}_{\Xlift}$.
\end{itemize}

\begin{notation}
We again let $x$ denote the map $\Xlift \to \PP^1_{\Zq}$.  Note that this reduces to $x : X \to \PP^1_k$ in the special fiber. 

Let $U$ (resp. $V$) denote the base change of $\Ulift$ (resp. $\Vlift$) to $\F_q$.  We write $\Xgen$ (resp. $\Ugen$ \ldots.) for the generic fiber of $\Xlift$ (resp. $\Ulift$ \ldots).
\end{notation}

\begin{example}[c.f.~{\cite[Example 2.4]{BT20}}]
Suppose $p$ is odd.  If $X$ is a hyperelliptic curve of genus $g$, take $\Qlift(x,y) = y^2 - f(x)=0$ where $\deg(f)$ is odd and has distinct roots modulo $p$.  Then $r(x) = f(x)$ and $U$ is the affine curve with the ramified points removed.  It is easy to check that a choice of the required integral bases is
\[
b^0_0 = 1, \quad b^0_1 = y, \quad \quad  b^\infty_0 = 1, \quad b^\infty_1 = y/x^{g+1}.
\]
The discriminant of $r(x)$ is in $\Zq^\times$ as the roots of $f$ are distinct modulo $p$, or equivalently the ramification in the generic and special fibers is the same.  We likewise see that the third condition in Assumption~\ref{assumption:nicelift} holds, and geometrically means that the branch points of the hyperelliptic curve are distinct modulo $p$.
\end{example}

\begin{remark} \label{remark:findinglift}
As discussed in \cite[\S2]{TuitmanII}, if $x$ is tamely ramified then Assumption~\ref{assumption:nicelift} holds.  Tuitman remarks that if $p \neq 2$ then $X$ is a tame cover of the projective line so if we are willing to vary the morphism $x$ (and plane curve model $Q$) the assumption can be satisfied.  More recent work also shows a tame map exists when $p=2$ \cite{klw23}.  
However it appears a subtle question to \emph{computationally find} a lift $\Qlift$ of a $Q$ representing $X$ satisfying Assumption~\ref{assumption:nicelift}.  See \cite{CT18,CV20} for some partial results.
\end{remark}

\subsection{Practical Setup}
We will end up performing certain computations over a number field rather than a $p$-adic field.  In particular, we require $\Qlift$ to have coefficients in the ring of integers in a number field with residue field $\F_q$, not simply in $\Zq$.  This is not a loss of generality as Tuitman's implementation \cite{pcc} of his algorithm already requires this.  There is no theoretical reason we must do so, but a variety of practical reasons make this beneficial.  

\begin{remark}\label{remark:overnumfield}
The main reason Tuitman chooses to work over a number field relates to the first part of Assumption~\ref{assumption:nicelift}, about having integral bases with nice properties.  There are efficient and implemented algorithms for computing integral bases \cite{hess02,bauch16} if $\Qlift$ has coefficients in a number field, but it is not clear how to do so $p$-adically and in any case no implementation exists: see \cite[Remark 2]{CT18}.  We need these bases as input to the main portion of the algorithm, and in practice if we pick $\Qlift$ well the bases as computed by Magma ``usually'' have nice properties (see Heuristic H \cite[Definition 1.2]{CV20}).  Thus we can often lift $Q$ and compute the nice integral bases automatically. 
\end{remark}

\begin{remark}
There are some practical issues with working over a number field instead of simply $\Q$, leading Tuitman to create two versions of his program for point counting, one over prime fields and one over general finite fields.  
For simplicity, we restrict our implementation to the first case when $k = \F_q = \F_p$.  In other words, we require that the polynomial $\Qlift$ defining $\Ulift$ actually has integer coefficients.  

The only subtle point for extensions of $\F_p$ is modifying the linear algebra computations to reflect that the Frobenius and Verschiebung are semilinear.  
\end{remark}

Other parts of the computation are approximations of $p$-adic computations which cannot be performed exactly.  This entails selecting a positive integer $\Nint$ to be the \emph{internal precision} and performing computations modulo $p^{\Nint}$.  In particular, many of the steps leading to computing a matrix for Frobenius are done with fixed internal precision.  As we will find in Section~\ref{sec:precision}, it is tricky to correctly reason about the $p$-adic precision of the results of these computations.

\section{Cohomology Theories} \label{sec:cohomology}

In this section we review several $p$-adic cohomology theories, their relationships, and how they are used in the algorithm.  We do not aim for maximum generality, simply for results which apply to the smooth and proper $X/\F_q$ or $\Xlift/\Zq$ from Section~\ref{sec:situation}.  A general reference, focused on applications to point counting, is \cite{kedlaya08}.

\subsection{Rigid Cohomology} \label{ss:rigid}
See \cite{lestum07} for general background on rigid cohomology.  For our purposes, the main space to be used is the first rigid cohomology $\Hrig^1(X)$, which is a $\Qq$-vector space of dimension $2 g(X)$. 
Most of the computations actually take place in the larger $\Hrig^1(U)$: $\Hrig^1(X)$ is the kernel of the sum of the residue maps at the points in $X-U$ (see \cite[Theorem 3.11]{TuitmanII}).  The space $\Hrig^1(U)$ may be realized as the  cokernel of the map $d : \Rdagger \to \Omega^1(\Ugen) \otimes  \Rdagger$, where $ \Rdagger$ is the ring of overconvergent functions on $\Ulift$.   This alternate description makes it easy to lift the Frobenius and determine the Frobenius action on rigid cohomology.  
Tuitman's algorithm computes (a $p$-adic approximation to) this map in order to count points on $X$.  We will build on this capacity.

\begin{remark} \label{remark:Espaces}
Tuitman provides an explicit description \cite[Theorem 3.6]{TuitmanII}: 
\begin{equation} \label{eq:E0Einfty}
\Hrig^1(U) \cong \frac{E_0 \cap E_\infty}{d (B_0 \cap B_\infty)}
\end{equation}
where $E_0 \cap E_\infty$ and $d(B_0 \cap B_\infty)$ are finite-dimensional $\Qq$-vector spaces of differentials with explicit forms describing their behavior above ramified points and infinity.  This representation is used internally in the implementation, but the exact details need not concern us.  We will later just need the following two ideas:
\begin{enumerate}
\item  using this representation, the implementation of Tuitman's algorithm computes a basis $\Brig$ for $\Hrig^1(X)$.  The $\Zq$-span of these basis elements is a $\Zq$-lattice $\Lrig$ in $\Hrig^1(X)$. 
\item  given an element of $\Omega^1(\Ugen) \otimes \Rdagger$ representing a rigid cohomology class, we can compute an alternate representative from  $E_0 \cap E_\infty$.  
\end{enumerate}
The second works by replacing the differential with an equivalent one with better behavior at the ramified points, then doing the same to get better behavior at infinity so that it lies in $E_0 \cap E_\infty$: see \cite[Theorem 3.7, 3.8]{TuitmanII}.  Tuitman also computes the residue maps so we can recover $\Hrig^1(X)$.
\end{remark}

The two major remaining steps to count points are to approximate a lift of Frobenius to $\Rdagger$ and then to 
use it to approximate the matrix of Frobenius on $\Hrig^1(U)$ and then $\Hrig^1(X)$.  Note that applying the Frobenius to a differential in $E_0 \cap E_\infty$ will not produce a differential of this form, so the ability to find representatives for rigid cohomology classes using the second point above is essential.

\subsection{Algebraic de Rham Cohomology} \label{ss:derham}

Let $C$ be a smooth proper curve over a base $S$.  (We will need $S$ to be either the spectrum of a field or a DVR.)  The de Rham complex is
\[
0 \to \cO_{C/S} \overset{d} \to \Omega^1_{C/S} \to 0
\]   
and the de Rham cohomology $\HdR^\bullet(C/S)$ is the hypercohomology, which may be computed using any acyclic resolution.
We drop $S$ from the notation when it is clear from context.

 The \v{C}ech resolution is a standard acyclic resolution, but we will find it more convenient to use an alternate resolution expanding on the classical description of de Rham cohomology in terms of differentials of the second kind.  Some version of this idea has been known to experts for a while, often expressed differently: the details over a field were recently written down in \cite[\S2]{bgkm}. 

\begin{definition}  \label{defn:secondkind}
For a closed point $Q$ of $C$, let $\widehat{\mathcal{O}}_{C,Q}$ denote the completed local ring at $Q$, and let $K_Q$ be its field of fractions.  
        An enhanced differential of the second kind on $C$ is a pair $(\omega,(f_Q)_{Q \in C})$ where:
\begin{enumerate}
	\item $\omega$ is a (rational) differential on $C$;
	\item for each $Q \in C$, $f_Q \in K_Q/\widehat{\mathcal{O}}_{C,Q}$ and if $\tilde{f}_Q \in K_Q$ represents $f_Q$ then $\omega-d\tilde{f}_Q \in \Omega^1_{\widehat{\mathcal{O}}_{C,Q}/S}$. 
\end{enumerate} 
\end{definition}

Over a field of characteristic zero, the $(f_Q)_{Q \in C}$ can be recovered from the local expansions of $\omega$ via taking an antiderivative, provided all the residues of $\omega$ vanish.  This recovers the classic notation of differentials of the second kind.  

We can generalize the classic description of de Rham cohomology as differentials of the second kind modulo exact differentials.
For a function $f$ on $C$, we let $d_C(f) = (df,(f)_{Q \in C})$ and note the image is an enhanced differentials of the second kind.  We call the image of $d_C$ the exact enhanced differentials of the second kind.

\begin{fact}[see {\cite[Theorem 2.6]{bgkm}}]
\label{fact:differentialssecondkind}
Let $S$ be the spectrum of a field or a DVR.
Then $\HdR^1(C/S)$ is isomorphic to the space of enhanced differentials of the second kind modulo exact ones.
\end{fact}

We will explain the proof in Section~\ref{ss:derham}, which also gives a more concrete realization.  Pick an effective divisor $D$ on $C$ and a positive integer $n$ with $\deg(nD) > \max(2g(C)-2,0)$.  Then to compute $\HdR^1(C/S)$ it suffices to work with enhanced differentials of the second kind $(\omega,(f_Q)_{Q \in C})$ where $\omega \in H^0(\Omega^1_{C/S}((n+1)D))$ and $f_Q =0$ if $Q \not \in D$ and $\ord_Q(f_Q) \geq - \ord_Q(nD)$ if $Q \in D$.  We therefore omit the $f_Q$ with $Q \not \in D$ from the notation.
Classes $(\omega,(f_P)_{P \in D})$ and $(\nu,(g_P)_{P \in D})$ are equivalent if there is $h \in H^0(\cO_{C/S}(nD))$ such that
\[
\omega - \nu = dh \quad \text{and} \quad f_P - g_P \equiv h \mod{\widehat{\mathcal{O}}_{C,P}} \textrm{ at each } P \in D.
\]

\begin{remark}
The advantages of this description of the de Rham cohomology are:
\begin{enumerate}
\item it avoids having to rewrite rational differentials (or functions) as differences of differentials (or functions) with poles at specific locations, as required by the \v{C}ech description;
\item  it is easy to connect with the description of rigid cohomology in terms of differentials;
\item  it is independent of the characteristic and even works over $\Zq$.
\end{enumerate}
\end{remark}

We end by comparing with rigid cohomology.  
 We adopt the notation of Section~\ref{sec:situation}, and take $D = \mathcal{D}_{\Xlift,\Qq}$.  
 
 \begin{fact} \label{fact:rigid}
Sending the class of $(\omega,(f_P)_{P \in D})$ to the class of $\omega$ induces an isomorphism $\HdR^1(\Xgen) \cong \Hrig^1(X)$. 
\end{fact}

More generally, there is an isomorphism $\HdR^1(\Ugen) \cong \Hrig^1(U)$: see \cite[Theorem 3.2]{TuitmanII}, which follows from the algorithm described in \cite[\S3]{TuitmanII} or as a special case of a general result of Baldassarri and Chiarellotto \cite{BC94}.  Our result follows as $\Hrig^1(X)$ (resp. $\HdR^1(\Xgen)$) is the kernel of sums of residue maps: see \cite[Theorem 3.11]{TuitmanII} (resp. the excision sequence for de Rham cohomology \cite[\S1.6]{kedlaya08}).

\subsection{Crystalline Cohomology}
The crystalline cohomology $\Hcrys^1(X)$ of $X$ is a free $\Zq$-module of rank $2g(X)$.  The direct definition of crystalline cohomology is neither relevant nor useful for us: we  instead need the relation of crystalline cohomology with other cohomology theories and with $p$-divisible groups.

\begin{fact} \label{fact:crystallineisomorphisms}
There are natural isomorphisms $\Hcrys^1(X) \cong \HdR^1(\Xlift)$ and
\[
\Hcrys^1(X) \tensor{\Zq} \Qq \cong \HdR^1(\Xlift) \tensor{\Zq} \Qq  \cong \HdR^1(\Xgen) \cong \Hrig^1(X).
\]
\end{fact}

See \cite[\S2.1]{kedlaya08} for a review of the relation between crystalline and de Rham cohomology: the connection to rigid cohomology is Fact~\ref{fact:rigid}.  The crystalline cohomology of $X$ naturally has
a Frobenius map $F$ induced by the functoriality.  The adjoint map $V$ satisfies $FV=VF=p$  and is known as the Verschiebung.  
Note that we then obtain similar maps on $\HdR^1(\Xlift)$, without lifting the Frobenius map of schemes to characteristic zero.  

\begin{remark}
Crystalline cohomology provides a lattice in $\Hrig^1(X)$ which is stable under Frobenius and Verschiebung.  More concretely, the matrices representing $F$ and $V$ with respect to a basis of $\Hcrys^1(X)$ have entries in $\Zq$, while the matrices with respect to a basis of $\Hrig^1(X)$  will usually have entries with negative valuation.  This is crucial for understanding the $p$-adic precision required in various steps of our algorithm.  
\end{remark}

We now turn to the connection with $p$-divisible groups.  Recall that the Dieudonn\'{e} ring is the ring $W(k)[F,V]$ and the indeterminates $F$ and $V$ satisfy the relations
\[
FV = VF = p, \quad F c = \sigma(c) F, \quad \text{and} \quad c V = V \sigma(c)
\]
where $\sigma: W(k) \to W(k)$ is the unique automorphism reducing to the Frobenius.    A \emph{Dieudonn\'{e} module} is a module over the Dieudonn\'{e} ring.   For a short guide to Dieudonn\'{e} modules and their relation with $p$-divisible groups, see \cite[\S1.4.3]{cco14}.  The key result is a Dieudonn\'{e} functor $D$ giving an anti-equivalence of categories.

In our situation, $W(k) = \Zq$ and $\Hcrys^1(X)$ is naturally a Dieudonn\'{e} module using the Frobenius and Verschiebung.

\begin{fact} \label{fact:Dieudonne}
Let $\Jac(X)[p^\infty] $ be the $p$-divisible group of the Jacobian of $X$.
There is an isomorphism of Dieudonn\'{e} modules $D(\Jac(X)[p^\infty]) \cong \Hcrys^1(X)$.
\end{fact}

This is essentially Mazur-Messing \cite{MM74}, plus the isomorphism between the first crystalline cohomology of a curve and of its Jacobian.

\begin{remark}
Similarly, $D(\Jac(X)[p^n]) \cong \Hcrys^1(X) \tensor{\Zq} \Zq/(p^n)$.
\end{remark}

\section{Computing the de Rham Cohomology of the Lift} \label{sec:computing cohomology}

\subsection{More on Enhanced Differentials of the Second Kind} \label{ss:effectivesecondkind}
We now explain how to compute a basis for $\HdR^1(\Xlift)$ using enhanced differentials of the second kind. We begin by reviewing the proof of Fact~\ref{fact:differentialssecondkind} and realizing it can be made effective.  

\begin{remark}
As previously mentioned, versions of this idea have been known to experts for a while.  Something similar appears in notes of Edixhoven \cite[Theorem 5.3.1]{edixhoven} for hyperelliptic curves, with details in the thesis of van der Boogart \cite{bogaart}.  A version also appears when Tuitman is  analyzing $p$-adic precision \cite[Proposition 4.5]{TuitmanII}.
\end{remark}

We let $C/S$ be our smooth proper curve and $D$ be a (relative) effective divisor on $C$.  Pick $n$ such that $\deg(nD) > \max(2g(C)-2,0)$.  We assume that $S$ is the spectrum of a field or a DVR.  The key idea is to compute de Rham cohomology using the acyclic ``pole-order'' resolution in Figure~\ref{fig:acyclic resolution}.

\begin{figure}[ht]
\[\begin{tikzcd}
	0 & {\Omega^1_{C/S}} & {\Omega^1_{C/S}((n+1)D)} & {\Omega^1_{C/S}((n+1)D)|_{(n+1)D}} & 0 \\
	0 & {\mathcal{O}_{C/S}} & {\mathcal{O}_{C/S}(nD)} & {\mathcal{O}_{C/S}(nD)|_{nD}} & 0 
	\arrow[from=1-1, to=1-2]
	\arrow[from=1-2, to=1-3]
	\arrow["\imath_2",from=1-3, to=1-4]
	\arrow[from=1-4, to=1-5]
	\arrow[from=2-1, to=2-2]
	\arrow["d"', from=2-2, to=1-2]
	\arrow[from=2-2, to=2-3]
	\arrow["d"', from=2-3, to=1-3]
	\arrow["\imath_1",from=2-3, to=2-4]
	\arrow["d"',from=2-4, to=1-4]
	\arrow[from=2-4, to=2-5]
\end{tikzcd}\]
\caption{The Pole-Order Resolution $\cD(n)$ of the de Rham Complex.}
\label{fig:acyclic resolution}
\end{figure}

By our choice of $n$, it is standard that the resolution is acyclic using Serre duality when working over a field.  This continues to hold when $S=\Spec \Zq$, and furthermore the global sections of the sheaves involved are free, as explained in \cite[Ch. 3, 3.1. Lemma]{bogaart}.  The same argument works for the spectrum of any discrete valuation ring (and in fact even more generally).  Thus we extend \cite[Theorem 2.6]{bgkm}:

\begin{theorem}
Let $S$ be the spectrum of a field or a DVR.  Then $\HdR^1(C/S)$ is isomorphic to the first homology of the total complex of global sections of the pole-order resolution in Figure~\ref{fig:acyclic resolution}.
\end{theorem}

To make $\cO_{C/S}(nD)|_{nD}$ more concrete, note that $\cO_{C/S}(nD)|_{nD} \cong \cO_{nD}(nD)$ is a skyscraper sheaf supported on $D$. 
Given a point $P$ in the support of $D$ with multiplicity $m_P$, let $A_P$ denote the local ring at the special point of $P$, $R_P$ be the residue ring, and $t_P$ a uniformizer at $P$ (a generator of the ideal sheaf $\cO_{C/S}(-P)$ in a neighborhood of $P$).  By smoothness, we know that the completion of $A_P$ is isomorphic to $R_P\powerseries{t_P}$.  Then the stalk of $\cO_{C/S}(nD)|_{nD}$ at $P$ is $t_P^{-n m_P} A_P / A_P \cong t_P^{-n m_P} R_P\powerseries{t_P} / R_P\powerseries{t_P}$.  Thus the map $\imath_1: \cO_{C/S}(nD) \to \cO_{C/S}(nD)|_{nD}$ records the 
``Laurent series tails'' of rational functions at the points in the support of $D$.  These are the $(f_P)_{P \in D}$ appearing in the notion of enhanced differentials of the second kind. 
Similarly the map $\imath_2: \Omega^1_{C/S}((n+1)D) \to \Omega^1_{C/S}((n+1)D)|_{(n+1)D}$ consists of taking ``Laurent series tails'' of expansions at each $P \in D$.

\begin{notation}
Let $S_{nD}$ (resp. $S'_{(n+1)D}$) denote the global sections of $\cO_{C/S}(nD)|_{nD}$ (resp. $\Omega^1_{C/S}((n+1)D)|_{(n+1)D}$), so:
\begin{align*}
S_{nD} & \cong \prod_{P \in \Supp(D)} \left( t_P^{-n m_P} R_P\powerseries{t_P} \right)/ R_P\powerseries{t_P} \\
S'_{(n+1)D} & \cong \prod_{P \in \Supp(D)} \left( t_P^{-(n+1) m_P} R_P\powerseries{t_P} \right)/ R_P\powerseries{t_P}  \cdot dt_P.
\end{align*}
\end{notation}

\subsection{Computing with Enhanced Differentials of the Second Kind}
We now present Algorithm~\ref{alg2}, an idealized algorithm for computing a basis for $\HdR^1(C/S)$ using the total complex %
\begin{equation} \label{eq:totalcomplex}
H^0(C, \cO_{C/S}(nD)) \overset{d \oplus \imath_1} \to H^0(C,\Omega^1_{C/S}((n+1)D)) \oplus S_{nD} \overset{d - \imath_2} \to S'_{(n+1)D}.
\end{equation}

\begin{algorithm} 
\caption{Idealized Algorithm for de Rham Cohomology}
\label{alg2}
\begin{algorithmic}[1]
    \STATE \label{step1} Compute bases for $H^0(C,\Omega^1_{C/S}((n+1)D))$ and $H^0(C,\cO_{C/S}(nD))$.
    \STATE \label{step2} Compute the uniformizers $t_P$ at each $P \in \Supp(D)$. 
       \STATE \label{step3} Compute the images of basis elements %
under $\imath_1$ and $\imath_2$ by computing local expansions.
\STATE  \label{step4} Compute $\ker (d - \imath_2)$ and $\Image(d \oplus \imath_1)$.
\STATE  \label{step5} Compute a basis for the quotient $\ker (d - \imath_2)/\Image(d \oplus \imath_1)$, which is isomorphic to $\HdR^1(C/S)$.
\end{algorithmic}
\end{algorithm}

If $S$ is the spectrum of a field, the first three steps of the algorithm (computing Riemann Roch spaces and spaces of differentials, computing uniformizers, and computing local expansions) are well studied problems and efficient algorithms have been implemented (at least for exact coefficient fields).  The fourth and fifth steps then can be easily solved using linear algebra.  However, there are two complications with making this work over $\Zq$ in order to compute $\HdR^1(\Xlift/\Zq)$:
\begin{itemize}
\item  $\Zq$ is not a field and, to the best of our knowledge, algorithms for the geometric problems in the first three steps have only been studied over fields;
\item  $\Zq$ (and $\Qq$) are not exact, so any implementation would have to correctly track $p$-adic precision.
\end{itemize}

When implementing the algorithm, we work over a number field.  This is no loss of generality: we need a model of $\Xlift$ over a number field anyway to use Tuitman's work \cite{pcc}.  This choice allows us to use Magma's function field machinery over the field of fractions which does not work over $p$-adic fields.  The choice also removes the need to perform linear algebra over $\Zq$ and track $p$-adic precision.  We will discuss the implementation in Section~\ref{ss:implementation}: it gives a practical (but not theoretically pleasing) solution to both complications.

For the rest of this subsection, we focus on analyzing a portion of the idealized algorithm treating the first two steps as a black box.  The geometric questions for curves over a DVR that we ignore are interesting but beyond the scope of this work.

We temporarily let $\cO \subset \Zq$ denote the ring of coefficients we will work with: for a theoretical analysis we take $\cO = \Zq$ but the implementation will take $\cO$ to be the ring of integers in a number field.  Let $K$ be the field of fractions.

\begin{notation}
Viewing $\Xlift$ as a cover of $\PP^1$ via projecting onto the $x$-coordinate, let $D = \sum_{i=1}^t n_i P_i$ be the divisor above infinity (over $\cO$).
Let 
\[
s = \rank_{\Zq} H^0(\Xlift,\cO_{\Xlift}(nD)) \quad \text{and}\quad s' = \rank_{\Zq} H^0(\Xlift,\Omega^1_{\Xlift}((n+1)D)).
\]  
\end{notation}

\begin{remark} \label{remark:growth}
Recall that the genus of $\Xlift$ is $O(d_x d_y)$ \cite[Proposition 4.1]{TuitmanII} so that $\deg(nD)$ can also be chosen to be $O(d_x d_y)$.  
We therefore conclude that $s$ and $s'$ are $O(d_x d_y)$ by Riemann-Roch.  Finally, notice that   $t$ is at most $d_y$ (the degree of the projection map).
\end{remark}

\begin{theorem} \label{theorem:computingderham}
Suppose we are given:
\begin{itemize}
\item  differentials $\omega_1,\ldots, \omega_{s'}$ (resp. functions $f_1,\ldots,f_{s}$) defined over $\cO$ forming an $\Zq$-basis for $H^0(\Xlift,\Omega^1_{\Xlift}((n+1)D))$ (resp. $H^0(\Xlift,\cO_{\Xlift}(nD))$);

\item  functions $u_1, \ldots, u_t$ on $\Xlift$ defined over $\cO$, one of which is a uniformizer at each element of $\Supp(D)$.
\end{itemize}
Then we may compute a $\Zq$-basis of $\HdR^1(\Xlift)$ where the number of operations in the coefficient ring $\cO$ is polynomial in $d_x$ and $d_y$.
\end{theorem}

\begin{proof}
We slightly modify steps \ref{step3}-\ref{step5} of Algorithm~\ref{alg2}.

For the third step, we compute local expansions over $K$, since the functions, differentials, and uniformizers are defined over $\cO$.  The coefficients of the local expansions will lie in the residue fields of $P_1,\ldots,P_t$ which are extensions of $K$.  %
As one of the $u_1, \ldots, u_t$ is a uniformizer at each element of $\Supp(D)$, we conclude that the coefficients of the local expansions are integral over $\cO$ (i.e. the coefficients of the expansion lie in the residue ring of the element of $\Supp(D)$, which is a finite extension of $\cO$).  Therefore the denominators of the coefficients are prime to $p$.  After rescaling the chosen bases, we therefore may assume the coefficients are integral over $\cO$ without changing the $\Zq$-span.

Using the rescaled bases of the spaces of functions and differentials, and the natural bases for $S_{nD}$ and $S'_{(n+1)D}$ given by powers of the uniformizers, we may therefore write down matrices which represent $\imath_1$ and $\imath_2$.  %
It is straightforward to write down a similar matrix for $d$ by differentiating.  

Using this information, the computation of bases for $\ker(d - \imath_2)$ and $\Image ( d \oplus \imath_1)$ in the fourth step is then a computation using Hermite normal form over $\cO$. Finally, computing the quotient in step~\ref{step5} can be done using Smith normal form.
See Cohen's textbook for a discussion of Hermite and Smith normal forms over Dedekind domains \cite[\S1.3,1.4,1.5]{cohenadvanced} and their connection with the computations of bases for kernels and quotients.

It remains to analyze the complexity and show the number of operations performed is polynomial in $d_x$ and $d_y$.  
At each of the $t$ elements $P \in \Supp(D)$, we must compute at most $(n+1) \ord_P(D)$-terms in the local expansion at $P$.  As $\deg(nD)$ is $O(d_x d_y)$, so are $\dim S'_{(n+1)D}$ and $\dim S_{nD}$ and hence so is the total number of terms in the local expansions we must compute.  Thus a polynomial number of operations in $\cO$ suffice to compute matrices representing $\imath_1$ and $\imath_2$.

For the computations of the Hermite and Smith normal forms, even the simplest versions of the algorithms only require a polynomial number of operations in the dimensions of the matrices, which as seen above are $O(d_x d_y)$.  
\end{proof}

\begin{remark}
As is common in algorithms based on linear algebra, we analyze the complexity based on the number of operations in $\cO$.  This can be somewhat misleading when working over an exact coefficient ring as the computation can easily introduce ``coefficient explosion'' in intermediate steps.  In that case, the running time would not actually be polynomial in the size of the inputs.  More sophisticated versions of the algorithms for computing the normal forms (also studied in \cite{cohenadvanced}) take this into account.

The idealized version of the algorithm works over $\Zq$, so in some sense coefficient explosion is not an issue.  However, computations over $\Zq$ inherently require approximation so this comes at the expense of having to track precision.
\end{remark}

\begin{remark}
If $K$ is a number field of degree $r$ over $\QQ$, then a single operation in $\cO$ requires a number of operations with integers which is polynomial in $r$.  If $K = \Qq$, then a single operation in $\cO=\Zq$ performed with precision $N$ requires a polynomial in $\log(q) N = \log(p) r N$ operations.
\end{remark}

\subsection{Implementation and Assumptions} \label{ss:implementation}

We conclude by describing a practical variant of Algorithm~\ref{alg2}.  

\begin{notation}
Let $\cO$ be the ring of integers of the number field $K$, the ring over which $\Xlift$ is defined.
Fix a place
$\fp$ above $p$ such that the complete local ring of $\cO$ at $\fp$ is isomorphic to $\Zq$.  Without loss of generality, we may assume that $[\F_q : \Fp] = r = [K:\Q]$, i.e. that $p$ is inert.
\end{notation}

\begin{remark}
The support of the divisor $D$ over infinity may look different when working over $\cO$ versus $\Zq$, as points may split when passing to the completion.
\end{remark}

In practice, we implement the linear algebra in Theorem~\ref{theorem:computingderham} over $\cO$ with a few small modifications.  We pick a $\Z$-basis for $\cO$ (and similarly for the residue rings of elements in the support of $D$) which reduces all of the Hermite and Smith normal form computations over the Dedekind domain $\cO$ to computations over $\Z$.  In fact, we can do everything using Magma's functionality for free $\Z$-modules without explicitly referencing the normal forms.

The main challenge is producing the functions, differentials, and uniformizers for use in Theorem~\ref{theorem:computingderham}.  The pragmatic answer is to simply perform computations over the number field $K$: the results will almost always work over $\Zq$ in practice.

Correctness can easily be verified.  To check that a set of functions (resp. differentials) defined over $K$ are a basis for the $\Zq$-module, the first requirement is that we can reduce them modulo $p$.  If we can, then using Nakayama's lemma we simply verify they form a basis for the corresponding space modulo $p$. 
Similarly, we can compute a uniformizer over $K$ and check its behavior in the special fiber.  But we do not have a guarantee this will always work.  

\begin{assumption} \label{assumption:funcdif}
We assume we are given differentials $\omega_1,\ldots, \omega_{s'}$ and functions $f_1,\ldots,f_{s}$, and uniformizers $u_1,\ldots,u_t$ satisfying the hypotheses of Theorem~\ref{theorem:computingderham}, for example those usually produced by Magma working over the number field $K$.
\end{assumption}

\begin{remark}
The algorithm for computing Riemann-Roch spaces implemented in Magma is based on \cite{hess02}, and requires working over a field.  Hess's algorithm involves computing integral closures in function fields. In order to use Tuitman's algorithm, we have already had to assume Assumption~\ref{assumption:nicelift}, one piece of which concerns having bases over $\Zq$ giving integral bases in the generic and special fibers.  Recalling Remark~\ref{remark:overnumfield}, we are ``usually'' able to find the required bases by working over $K$.  Thus Assumption~\ref{assumption:funcdif} is an extension of the preexisting assumption: integral bases for the function fields over $K$ and $\F_q$ behave similarly so the results of Hess's algorithm over $K$ give a $\Zq$-basis for the Riemann-Roch space.
\end{remark}

\begin{remark} \label{remark:RRpoltime}
The functions and differentials needed in Theorem~\ref{theorem:computingderham} can be computed in polynomial time when working over $K$.  The exact complexity exponent is not completely clear.  To the best of our knowledge it has not been analyzed for Hess's algorithm, which is what we are making use of in Magma.
\end{remark}

We illustrate Assumption~\ref{assumption:funcdif} for Artin-Schreier curves, showing it is reasonable.  
We look at $X$ given by $y^p - y = f(x)$, where $f(x) \in \F_q[x]$ is a polynomial of degree $d$ with $p \nmid d$.  We may take $Q(x,y) = y^p - y - f(x)$ as the plane curve model.  Suppose we are given $\Qlift = y^p -y - \widetilde{f}(x)$ and $\Xlift$ as in Section~\ref{sec:situation}.  

As always, we project onto the $x$-coordinate. We know the special fiber is (wildly) totally ramified only over infinity, and has genus $(p-1)(d-1)/2$.  The generic fiber also has genus $g=(p-1)(d-1)/2$, and is (tamely) totally ramified over infinity.  It also has
$(p-1)d$ simple branch points, the geometric points on $\Xgen$ where $py^{p-1}-1=0$.  

\begin{notation}
Let $P_\infty$ be the unique point at infinity, viewed as a $\Zq$-point.  For integers $i,j$ with $0 \leq j \leq p-1$ we let
\[
f_{i,j} \colonequals x^i y^j, \quad  \omegazero \colonequals \frac{dx}{1-py^{p-1}}, \quad \text{and} \quad  \omega_{i,j} \colonequals f_{i,j} \omegazero.
\]
\end{notation}

\begin{lemma} \label{lemma:ordersofvanishing}
For integers $i,j$ with $0 \leq j \leq p-1$, we have
\[
\ord_{P_\infty} (f_{i,j}) = -p i - dj \quad \text{and} \quad \ord_{P_\infty} \left(\omega_{i,j}\right) = 2g-2-pi - dj .
\]
These have poles at points above $0$ if and only if $i <0$, and are otherwise regular.
\end{lemma}

\begin{proof}
We just need a few elementary observations.
\begin{enumerate}
\item  %
The function $x$ has a pole of order $p$ at $P_\infty$ and is otherwise regular.  

\item  The function $y$ has a pole of order $d$ at $P_\infty$, and is otherwise regular. %

\item  The differential $\omegazero$
is regular and only has a zero of order $2g-2 = (p-1)(d-1)-2$ at $P_\infty$.  
\end{enumerate}
The result is now clear.
\end{proof}

\begin{remark}
Note that the differential $dx$ over $\Zq$ is not as well behaved.  In the generic fiber, it has simple zeros at the $(p-1)d$ geometric branch points and a pole of order $p+1$ at $P_\infty$.  In the special fiber, it has a zero of order 
$2g-2 = (p-1)(d-1)-2$ at $P_\infty$ due to the wild ramification in characteristic $p$.
\end{remark}

We take $n=(2g-1)$ and $D = [P_\infty]$ and investigate Assumption~\ref{assumption:funcdif}.

\begin{proposition} \label{prop:artinschreierassumptions}
A $\Zq$-basis for $H^0(\Xlift,\cO_{\Xlift}(nD))$ is
\[
\Bone = \left \{ f_{i,j} \colonequals x^i y^j : 0 \leq i , \, 0 \leq j \leq p-1, \, pi + dj \leq (p-1)(d-1)-1 \right \}.
\]
A $\Zq$-basis for $H^0(\Xlift,\Omega^1_{\Xlift}((n+1)D))$ is 
\[
\Btwo = \left \{ \omega_{i,j} \colonequals x^i y^j \frac{dx}{1-py^{p-1}}: 0 \leq i, \, 0 \leq j \leq p-1, \, pi+dj \leq 2 (p-1)(d-1) -2 \right \}.
\]
Pick integers $a,b$ such that $pa+db=-1$: then $u=x^a y^b$ is a uniformizer at $P_\infty$.
\end{proposition}

\begin{proof}
Both $\Zq$-modules are free of known dimension as explained in \cite[Ch. 3, 3.1. Lemma]{bogaart}. Lemma~\ref{lemma:ordersofvanishing} shows the specified elements lie in the appropriate spaces.  They are linearly independent as each has distinct order of vanishing at the point at infinity, and reducing modulo $p$ certainly form a basis for the analogous spaces in the special fiber.  Thus they form a basis by Nakayama's lemma.

Lemma~\ref{lemma:ordersofvanishing} also shows that $u$ is a uniformizer at $P_\infty$: such $a,b$ exist since $\gcd(d,p)=1$.
\end{proof}

\begin{example}
Consider the case that $p=7, d=3$, and $\Qlift(x,y) = y^7-y - x^3$.  Thus $n = (d-1)(p-1) -1 =11$.  Working over $\Q$, the algorithm in Magma produces a basis 
\[
y^3,\, y^2,\, xy,\, y,\, x,\, 1
\]
for $H^0(\Xgen,\cO_{\Xgen}(nP))$, which is the natural one occurring in Proposition~\ref{prop:artinschreierassumptions}.  The basis for $H^0(\Xgen,\Omega^1_{\Xgen}((n+1)D)$ computed by Magma initially looks much more complicated when written as a multiple of $dx$, but we see they are monomial multiples of $\omegazero$ as in Proposition~\ref{prop:artinschreierassumptions}. 
\end{example}

\begin{remark}
We will briefly comment on how Proposition~\ref{prop:artinschreierassumptions} could generalize to arbitrary plane curves.  There are alternate approaches to computing Riemann-Roch spaces which rely on some form of Brill-Noether theory and use the notion of an adjoint curve or adjoint divisor.  See \cite{abcl22} for a recent example, and Section 1.3 of loc. cit. for references to this strand of inquiry.

For example, for a smooth projective curve $X$ with plane curve model given by $Q(x,y)=0$, a natural basis for the space of regular differentials is given by
\[
x^i y^j \frac{dx}{Q_y}
\]
where $Q_y$ denotes a partial derivative with respect to $y$.  The allowable indices $i$ and $j$ depend on the singularities of the plane model in the form of the adjoint curve, and can be modified to allow constrained poles above infinity.  This is used, for example, in \cite{SV} to compute a matrix representation of the Cartier operator on the space of regular differentials.  Note that $Q_y = p y^{p-1}-1$ in the case of an Artin-Schreier curve, so this matches the basis $\Btwo$ in Proposition~\ref{prop:artinschreierassumptions}.  One can imagine generalizing that proposition and implementing an algorithm for computing Riemann-Roch spaces over $\Zq$ by studying the adjoint curve and extending this approach to work in the relative setting.  But this is beyond the scope of this work.
\end{remark}

We also work out a basis for the crystalline cohomology of Artin-Schreier curves.  It is of independent interest and used in a test case for our implementation.  We use enhanced differentials of the second kind supported on $P_\infty$, i.e. coming from the total complex of Equation~\ref{eq:totalcomplex} with $D = [P_\infty]$ and $n=(2g-1)$.  Thus we only indicate the $P_\infty$ component in Definition~\ref{defn:secondkind}.  This generalizes \cite[Lemma 3.1]{SZ03} (a basis for $H^0(\Xlift,\Omega^1_{\Xlift})$) and \cite[Proposition 2.3]{Zhou21} (a basis for $\HdR^1(X/\F_q)$).  These use the \'{C}ech resolution, but can be converted using \cite[\S2.3]{bgkm}. 

\begin{theorem}
Let $h_{i,j} = (p-1-j)x \widetilde{f}'(x) - (i+1) p \widetilde{f}(x)$, and $h_{i,j}^{\geq i+2}(x)$ denote the terms of $h_{i,j}$ where the degree of $x$ is at least $i+2$.  We set
\[
\psi_{i,j} \colonequals -\frac{h_{i,j}^{\geq i+2}(x) y^{p-2-j}}{x^{i+2}} \frac{dx}{1 - py^{p-1}}.
\]
Then a basis for $\Hcrys^1(X) \cong \HdR^1(\Xlift)$ is 
\[
\left\{ (\omega_{i,j},0), (\psi_{i,j},f_{-(i+1),p-1-j} ) : 0 \leq i,j, \quad pi +dj \leq 2g-2 \,  \right\}.
\]
\end{theorem}

\begin{proof}
Note that $\omega_{i,j}$ is regular so $(\omega_{i,j},0)$ is a well-defined enhanced differential of the second kind.  Moreover $\psi_{i,j}$ is regular except at $P_\infty$.  Computing $d (f_{-(i+1),p-1-j})$ and using Lemma~\ref{lemma:ordersofvanishing} we see that $\psi_{i,j}-d(f_{-(i+1),p-1-j})$ is regular at $P_\infty$.  This computation is very similar to one in the proof of \cite[Proposition 2.3]{Zhou21}, and shows $(\psi_{i,j},f_{-(i+1),p-1-j})$ is well-defined.
All of these elements are defined over $\Zq$, and there are $2g = (p-1)(d-1)$ such elements, so it suffices to check their independence in the special fiber. 

It is well known that the regular $\omega_{i,j}$ form a basis for $H^0(X,\Omega^1_X)$ as they have distinct orders of vanishing at $P_\infty$.  The natural map 
\[
\HdR^1(X) \to \HH^1(X,\cO_{X}) \simeq (S_{nD}\otimes \F_p)/H^0(X,\cO_X(nD))\]
sends $(\omega,f)$ to $f$, and it is again standard (implicit in \cite[Proposition 2.3]{Zhou21} for example) that the $f_{-(i+1),p-1-j}$ form a basis. 
\end{proof}

\section{Analyzing \texorpdfstring{$p$}{p}-adic Precision} \label{sec:precision}

\subsection{Some Lattices}
The relative positions of several lattices in $\Hrig^1(X)$ are important for understanding the precision of Tuitman's algorithm.  Recall from Remark~\ref{remark:Espaces} that $\Brig$ is the basis for $\Hrig^1(X)$ used internally in Tuitman's implementation, and $\Lrig$ is the $\Zq$-lattice it generates.

\begin{notation}
Fix a $\Zq$-basis $\Bcrys$ for $\Hcrys^1(X)$, and let $\Lcrys$ be the image of $\Hcrys^1(X)$ in $\Hrig^1(X)$ under the inclusions and isomorphisms provided by Fact~\ref{fact:crystallineisomorphisms}.

Using the description in Remark~\ref{remark:Espaces}, the intersection $(E_0 \cap E_\infty) \cap \Omega^1(\Ulift)$ provides a natural $\Zq$-lattice in $\Hrig^1(U)$ and, after passing to the kernel of the residue maps, in $\Hrig^1(X)$.  We denote it by $\Lambda_{\rig}'$, and pick a basis $\Brig'$.

We denote the matrix of $F$ with respect to a basis $\Bcrys$ (resp. $\Brig$) by $[F]_{\crys}$ (resp. $[F]_{\rig}$), and similarly for $V$.
\end{notation}

The matrices for $F$ and $V$ relative to (any) basis for $\Hcrys^1(X)$ lie in $\Mat_{2g,2g}(\Zq)$ as $F$ and $V$ preserve the rank $2g$ $\Zq$-module $\Hcrys^1(X)$.  However, the matrices for $F$ (and $V$) relative to $\Brig$ will usually not be in  $\Mat_{2g,2g}(\Zq)$ and it is crucial but somewhat subtle to understand the valuation of the entries.

\begin{notation} \label{notation:gamma1gamm2}
Let $\gamma_1$ and $\gamma_2$ be the smallest non-negative integers such that
\begin{equation}
p^{\gamma_1} \Lrig \subset \Lcrys \subset p^{-\gamma_2} \Lrig.
\end{equation}
Let $\gamma_3$ be the smallest non-negative integer such that
\[
p^{\gamma_3} \Lrig' \subset \Lrig.
\]
\end{notation}

Basic linear algebra then shows that 
\begin{equation} \label{eq:Frig denominator}
[F]_{\rig} \subset p^{-\gamma_1 - \gamma_2} \Mat_{2g,2g}(\Zq).
\end{equation}

In contrast, Tuitman \cite[Definition 4.4, Proposition 4.5]{TuitmanII} defines explicit $\delta_1,\delta_2$ and proves that
\begin{equation} \label{eq:delta1delta2}
p^{\delta_1} \Lrig' \subset \Lcrys \subset p^{-\delta_2} \Lrig'.
\end{equation}
The matrix for $F$ relative to $\Brig'$ has valuation at least $-(\delta_1 + \delta_2)$.

\begin{remark} \label{remark:error}
The lattice $\Lrig'$ defined above is not equal to the lattice $\Lrig$ used in the implementation of Tuitman's algorithm, despite that being the intent.  This leads to a disconnect between the theoretical results from \cite{TuitmanII} (which are for $\Lrig'$), that are used to predict the precision of the computation, and the actual precision in Tuitman's implementation \cite{pcc}.  This is a source of occasional errors in the computation of the zeta function.  We will discuss an example in depth in Section~\ref{ss:precision example}.
\end{remark}

\subsection{How Much Precision is Needed?}
Computationally, all $p$-adic numbers are represented by integer or rational approximations.  Tuitman's algorithm approximates a lift of the Frobenius map on overconvergent functions and then the matrix $[F]_{\rig}$ for the Frobenius on the rigid cohomology of $X$.  All of these approximations are done modulo a fixed power $p^{\Nint}$: here $\Nint$ is the internal precision of Tuitman's algorithm.   We begin by reviewing a few ideas about $p$-adic precision.

As discussed in Remark~\ref{remark:Espaces}, we can rewrite elements of $\Bcrys$ as a linear combination of the elements of $\Brig$ using the inclusion from Fact~\ref{fact:crystallineisomorphisms}.  This also involves working modulo $p^{\Nint}$, and 
gives us a matrix $C \in \Mat_{2g,2g}(\Qq)$ representing the change of basis from $\Bcrys$ to $\Brig$.  Understanding how accurate the computed values for $[F]_{\crys}$ and $C$ are based on $\Nint$ is essential.

When approximating a $p$-adic number, the absolute precision is $a$ if the computed value is correct modulo $p^a$, i.e. if the valuation of the difference between the computed and correct values is known to be at least $a$.  Thus we speak of computing a $p$-adic number or $p$-adic matrix to precision $a$.  The precision (resp. valuation) of a $p$-adic matrix is the minimum precision (resp. valuation) of the entries.  

We will think about $p$-adic precision for matrices using the approach from \cite{crv14}, where Caruso, Roe, and Vaccon represent the uncertainty in an approximation as a lattice and understand how it changes via analyzing the derivative of the function being computed.  In their language, we are mainly interested in flat precision, where the lattice is $p^a \Mat_{2g,2g}(\Zq)$.  When applying a matrix with valuation at least $-b$ to the lattice, the resulting lattice is contained in $p^{a-b} \Mat_{2g,2g}(\Zq)$, so the result has (flat) precision at least $a-b$.  To illustrate the method, we repeat the proof of \cite[Proposition 3.6]{crv17} as we will need this result.

\begin{lemma} \label{lemma:precisioninverse}
Given $A \in \GL_n(\Qq)$ with precision $a$, suppose $v_p(A^{-1}) = m$.  Then the inverse $A^{-1}$ has precision at least $a+2m$.
\end{lemma}

\begin{proof}
The derivative of the matrix inversion map at $A$ sends $dA$ to $A^{-1} dA A^{-1}$.  If $v_p(dA)\geq a$ then $v_p(A^{-1} dA A^{-1}) \geq a+2m$ as desired.
\end{proof}

\begin{proposition} \label{prop:precision1}
Given $[F]_{\rig}$ with precision $N' \colonequals N + \gamma_1+\gamma_2+1$ and the change of basis matrix $C$ from $\Bcrys$ to $\Brig$ with precision $N'' \colonequals N + \gamma_1+2\gamma_2+1$, we obtain $[F]_{\crys}$ with precision $N+1$ and $[V]_{\crys}$ with precision $N$.
\end{proposition}

\begin{proof}
For ease of notation, let $A = [F]_{\crys}$ and $B= [F]_{\rig}$ so that $A= C^{-1}BC$.  The derivative of the conjugation map $(X,Y) \mapsto Y^{-1} X Y$ at $(B,C)$ is
\begin{equation} \label{eq:derivativeconjugation}
C^{-1} dC C^{-1} B C + C^{-1} dB C + C^{-1} B dC.
\end{equation}
Note that $v_p(C^{-1} B C) \geq 0$ as $A=[F]_{\crys}$ has integral entries.  By~\eqref{eq:Frig denominator}, we know $v_p(B) \geq -\gamma_1 - \gamma_2$ and by definition we know that $v_p(C) \geq - \gamma_2$ and that $v_p(C^{-1}) \geq - \gamma_1$.  Thus the valuation of \eqref{eq:derivativeconjugation} is at least 
\[\min(N'' - \gamma_1, N' - \gamma_1-\gamma_2, N''-\gamma_1-2\gamma_2)  = N+1.
\]
 Thus $A = [F]_{\crys}$ has precision $N+1$.

Now $[V]_{\crys} = p [F]_{\crys}^{-1}$ and $v_p([F]_{\crys}^{-1}) = v_p(p^{-1} [V]_{\crys})  \geq -1$ since $[V]_{\crys}$ has integral entries.  Thus Lemma~\ref{lemma:precisioninverse} shows that we obtain $[F]_{\crys}^{-1}$ with precision $N-1$ and hence $[V]_{\crys}$ with precision $N$.
\end{proof}

\begin{remark} \label{remark:precision1}
Similarly, it suffices to know a matrix for $F$ with respect to $\Brig'$ with precision $N + \delta_1 + \delta_2 + 1$ and change of basis matrix from $\Bcrys$ to $\Brig'$ with precision $N + \delta_1 + 2 \delta_2 + 1$.
\end{remark}

This allows us to compute $[F]_{\crys}$ and $[V]_{\crys}$ with any desired precision provided we correctly understand the precision of the outputs of Tuitman's algorithm.  %

\subsection{An Illustrative Example} \label{ss:precision example} As mentioned in Remark~\ref{remark:error}, there is a subtle mismatch between the theoretical analysis in \cite{TuitmanII} and its implementation \cite{pcc}.  We will give an example to demonstrate this issue.

\begin{example} \label{ex:precision}
Consider the Artin-Schreier curve $X$ given by $x^7-x = y^6 + 2y^3 + y^2  +3$ in characteristic $p=7$, with map to $\PP^1$ given by projecting onto the $x$-axis.  We compute $[F]_{\rig}$ and $C$ using Tuitman's algorithm with varying internal precision $\Nint$, and use this to compute $[F]_\crys$ and $[V]_\crys$.  The precision of the results are shown in Table~\ref{table:precision}.  (To see the precision, we also performed the computations with a much larger internal precision, in this case $\Nint=40$, and compared with this reference value.)  Note that increasing the internal precision increases the precision of the outputs as expected.

\begin{table}[hbt]
\begin{tabular}{|c||c|c|c|c|c|c|}
\hline
$\Nint$ & $[F]_{\crys}$ & $[V]_\crys$ & $C$ & $C^{-1}$ & $[F]_{\rig}$  \\
\hline
 19 &1 &0 &19 & 19 &8   \\ 
 20 &2 &1 &20 &20 &9    \\
 21 &3 &2 &21 &21 &10    \\
 22 &4 &3 &22 &22 &11    \\
 26 &8 &7 &26 &26 &15 \\
   \hline
\end{tabular}
\caption{Precision of outputs with varying internal precision in Tuitman's algorithm, $x^7-x = y^6 + 2y^3 + y^2  +3$ over $\F_7$} \label{table:precision}
\end{table}

The computations also reveal that 
\begin{equation}
v_7([F]_{\rig}) = -9, \quad v_7(C) = -9, \quad \text{and} \quad v_7(C^{-1}) =0.
\end{equation}
Recalling Notation~\ref{notation:gamma1gamm2}, this shows that $\gamma_1 = 0$ and $\gamma_2 = 9$.  
Note that $v_p([F]_\rig) \geq - \gamma_1 - \gamma_2 =-9$ in accordance with \eqref{eq:Frig denominator}. 

Using Proposition~\ref{prop:precision1}, to compute the crystalline cohomology of $X$ modulo $p$ (i.e. the de Rham cohomology of $X$ or equivalently the Dieudonn\'{e} module of $\Jac(X)[p]$), we take  $N = 1$ so that $N' = N''=11$.
We see that Tuitman's algorithm gives $[F]_{\rig}$ with enough precision when $\Nint \geq 22$.  

Note the isomorphism class of $\Jac(X)[p]$ or equivalently its Dieudonn\'{e} module $\HdR^1(X)$ is equivalent to the Ekedahl-Oort type of $X$.  Our result agrees with an independent computation of the EO type using code of Colin Weir \cite{weir_code} which computes $F$ and $V$ on $\HdR^1(X)$ using the \v{C}ech description of de Rham cohomology.
\end{example}

\begin{remark}
The precision of $C$ and $C^{-1}$ equals the internal precision.  This is a consequence of the fact we compute $\Bcrys$ and $\Brig$ exactly (i.e. working over $\Q$) instead of working with $p$-adic approximations.  For $\Bcrys$, this is a feature of how we are computing it.  We could instead approximate $\Brig$ to potentially speed up computations (Tuitman's implementation supports either approach), but this leads to errors and in any case finding a basis for $\Hrig^1(X)$  is not the bottleneck.
\end{remark}

Example~\ref{ex:precision} reveals two signs of a problem in the implementation of Tuitman's algorithm:
\begin{enumerate}
\item  The computation that $v_7([F]_\rig) = -9$ contradicts \cite[Corollary 4.6]{TuitmanII}, which asserts a lower bound on the valuation of $[F]_\rig$ which in this case is $-2$ (as $\delta_1=1$ and $\delta_2=1$).  Equivalently, the computed value $\gamma_2=9$ contradicts the second inclusion in \eqref{eq:delta1delta2} as $\delta_2=1$.

\item  The proof of \cite[Proposition 4.7]{TuitmanII} makes a prediction about the precision of $[F]_\rig$ based on $\Nint$ which is incorrect.  For example, when $\Nint= 20$ it asserts the precision of $[F]_\rig$ should be at least $17$, which is false.
\end{enumerate}

Both these are caused by the error noted in Remark~\ref{remark:error}, that the lattice $\Lrig'$ that is analyzed theoretically is not in fact equal to the lattice $\Lrig$ used internally in the implementation.

\subsection{Provably Correct Precision}
The precision errors cause errors when computing zeta functions. There are incomplete fixes in the implementation but no discussion of them in the literature, so we will explain this carefully.

 Running Tuitman's implementation \cite{pcc} on the curve in Example~\ref{ex:precision} and choosing to approximate $\Brig$ $p$-adically (the default choice), internal error-checking catches and reports an error with the numerator of the zeta function.\footnote{An initial approximation produced $v_p(\pr) = -42$, and then the internal precision was increased from $14$ to $99$ due to a partial understanding of the underlying issue.}

Performing the same computation and choosing to compute $\Brig$ exactly gives the correct zeta function and point count.  The most recent version of \cite{pcc} provably computes the zeta function, while the initial version and the version included in Magma use insufficient precision.\footnote{The incorrect versions choose internal precision $\Nint=23$ due to a partial understanding of the issue.  Then $[F]_\rig$ only had $p$-adic precision $12$, while precision $14$ is needed to provably compute the zeta function \cite[Proposition 4.9]{TuitmanII}.}

Let us review how to understand the precision of $[F]_\rig$, which appears in the proof of \cite[Proposition 4.7]{TuitmanII}.  The starting point is applying Frobenius to a differential $\omega_i \in \Brig$, obtaining
\begin{equation}
F(\omega_i) = \sum_{j \in \Z} \left ( \sum_{k=0}^{d_x-1} \frac{w_{i,j,k}(x)}{r^j} b^0_k \right) \frac{dx}{r} 
\end{equation}
where the $w_{i,j,k}(x) \in \Zq[x]$ satisfy $\deg(w_{i,j,k}(x)) < \deg(r)$ and have precision $\Nint$.  We next use idea (1) of Remark~\ref{remark:Espaces} to compute an element of $E_0 \cap E_\infty$ representing the same rigid cohomology class.  This causes a loss of precision as the rewriting process involves $p$-adic denominators.  When applied to terms with $j>0$ (resp. $j<0$), the loss in precision is $f_1(\Nint)$ (resp. $f_2$) using the notation from \cite[Definition 4.8]{TuitmanII}.  (The exact formulas for $f_1$ and $f_2$ are irrelevant for this explanation.)  Finally, this element gives an equivalence class in $E_0 \cap E_\infty / d(B_0 \cap B_\infty) \cong \Hrig^1(U)$, which up to the error in the approximation lies in $\Hrig^1(X)$, the kernel of the residue maps.   So we express it in terms of the basis $\Brig$ and hence obtain information about $[F]_{\rig}$.  We can express this via the composed quotient and projection maps, which is a linear map
\begin{equation} \label{eq:pr}
\pr : E_0 \cap E_\infty \to \Hrig^1(X).
\end{equation}

\begin{example}
For the Artin-Schreier curve in Example~\ref{ex:precision}, if $\Nint = 22$ then $f_1(\Nint) = 3$ and $f_2 = 2$.  Thus we would expect $[F]_\rig$ to have $p$-adic precision at least $22 - 3 = 19$.  However, when rewriting the equivalence class of the element of $E_0 \cap E_\infty$ in terms of $\Brig$, the coefficients actually have $p$-adic denominators.  In particular, the matrix for $\pr$  has valuation $-9$ (using a $\Zq$-basis for $(E_0 \cap E_\infty) \cap \Omega^1(\Ulift)$ and $\Brig$ as the bases).  This explains why $[F]_\rig$ has precision much lower than expected, in this case $11$ instead of $19$.
\end{example}

This last step should not cause a loss of precision if we use a basis for $\Hrig^1(X)$ that is a $\Zq$-basis for the lattice $\Lrig'$, as by definition $\Lrig'$ is the lattice induced by $\Omega^1(\Ulift) \cap (E_0 \cap E_\infty)$.  The fact that the matrix for $\pr$ has negative valuation illustrates that $\Lrig'$ is not equal to $\Lrig$.

\begin{lemma}
We have $v_p(\pr) \geq - \gamma_3$.
\end{lemma}

\begin{proof}
Let $\omega$ be an arbitrary element of $(E_0 \cap E_\infty) \cap \Omega^1(\Ulift)$.  Then image under the quotient map to $\Hrig^1(X)$ is in $\Lrig'$.  Recalling Notation~\ref{notation:gamma1gamm2}, we conclude the image lies in $p^{-\gamma_3}\Lrig$ as desired.
\end{proof}

There are two reasonable approaches to addressing the error regarding precision caused by $\Lrig$ and $\Lrig'$ differing.
\begin{enumerate}
\item  Correctly compute a basis for $\Lrig'$ and use that basis as $\Brig$ throughout the implementation of the algorithm.

\item  Use the current basis $\Brig$ for $\Lrig$ and account for the loss of precision caused by denominators in the entries of $\pr$.
\end{enumerate} 

Both are reasonable: we choose to implement the second in order to make this work more self-contained, treating Tuitman's algorithm and implementation as black boxes as much as possible. 

\begin{remark}
An additional reason to choose the second approach is that it is non-trivial to actually compute a basis for $\Lrig'$.  As explained in \cite[\S4.1]{TuitmanI}, this is a problem in linear algebra over $\Zq$, and can be approached using Smith normal form like some of the questions in Section~\ref{sec:computing cohomology}.  However, especially when working with approximations to $p$-adic numbers this is tricky to do, and an incorrect workaround in \cite{pcc} is what caused the underlying issue that $\Lrig \neq \Lrig'$.  

The fundamental mathematical problem is that for a matrix $A$ with coefficients in a domain $R$ with field of fractions $K$, computing a $K$-basis the kernel of $A$ does not necessarily give a $R$-basis for the kernel even if the coefficients of the basis elements lie in $R$.
\end{remark}

As discussed in Section~\ref{sec:computing cohomology}, in practice we compute $\Bcrys$ over a number field.  Tuitman's implementation has the option to either compute $\Brig$ over a number field or to compute a $p$-adic approximation: we choose the former option.  This is important as otherwise $\Brig$ and $v_p(\pr)$ can (and do) depend on the internal precision $\Nint$.  These precomputations allow us to know $\gamma_1,\gamma_2, \gamma_3$ (Notation~\ref{notation:gamma1gamm2}) before selecting an internal precision $\Nint$ and computing $[F]_\rig$, which is the most expensive step.

Recall that there were quantities $f_1(n)$ and $f_2$ defined in \cite[Definition 4.8]{TuitmanII} to measure the loss of precision upon rewriting a general differential to lie in $E_0 \cap E_\infty$  Then the discussion above shows:

\begin{proposition} \label{prop:precision2}
Let $N$ be a positive integer.
\begin{itemize}
\item If $\Nint \geq N + \max(f_1(\Nint),f_2) + \gamma_3$ then the $[F]_{\rig}$ computed by the implemented version of Tuitman's algorithm that computed $\Brig$ exactly has precision at least $N$.
\item  If $\Nint \geq N + \max(f_1(\Nint),f_2)$, then the idealized version of Tuitman's algorithm that computes a basis for $\Lrig$ exactly has precision at least $N$.
\end{itemize}
\end{proposition}

\begin{proof}
This follows from the above discussion: note for the first case that applying $\pr$ has the potential to decrease the precision by $v_p(\pr) \geq -\gamma_3$.  
\end{proof}

\begin{theorem} \label{thm:precision}
Let $N$ be a positive integer.  To compute $\Hcrys^1(X)$ with $p$-adic precision $N$ (equivalently, the group scheme $\Jac(X)[p^N]$), it suffices to pick an internal precision $\Nint$ which satisfies:
\begin{itemize}
\item $\Nint \geq N + \max(f_1(\Nint),f_2) + \gamma_1 + \gamma_2 + \gamma_3 + 1$ if using the implemented version of Tuitman's algorithm.
\item  $\Nint \geq N + \max(f_1(\Nint),f_2) + \delta_1 + \delta_2 + 1$ if using the idealized version.
\end{itemize}
\end{theorem}

\begin{proof}
Combine Proposition~\ref{prop:precision1} and Remark~\ref{remark:precision1}  with Proposition~\ref{prop:precision2}.
\end{proof}

\begin{remark} \label{remark:asymptotics}
Note that $f_1(x)$ is $O(\log(x))$, so it is easy to find $\Nint$ satisfying the inequality in Theorem~\ref{thm:precision}.  More precisely, assuming \cite[Assumption 2]{TuitmanII} and recalling \cite[Proposition 4.1, Definition 4.8]{TuitmanII} we see that:
\begin{align*}
f_1(x) & = O(\log(x) + \log(d_x) + \log(d_y)) \\ 
f_2 &= O( \log(d_x) + \log(d_y))\\
\delta_1,\delta_2 &= O( \log(d_x) + \log(d_y)).
\end{align*}
Thus when using the implemented version of Tuitman's algorithm we may choose $\Nint \in O(N \log(N)+ \log(d_x) + \log(d_y) + \gamma_1 + \gamma_2 + \gamma_3)$.  For the idealized version, we only need $\Nint \in O(N \log(N) + \log(d_x) + \log(d_y) )$.
\end{remark}

\begin{remark} \label{remark:don'tanalyzegamma}
To analyze our implementation, we would need a result on the growth rate of $\gamma_1,\gamma_2,\gamma_3$. This has little intrinsic mathematical interest as $\Lrig$ has little intrinsic mathematical meaning: it is the result of an incorrect attempt at computing the more natural $\Lrig'$.
\end{remark}

\section{Analysis of the Algorithm} \label{sec:analysis}

We analyze Algorithm~\ref{alg1}, tracking the assumptions and auxiliary data.  The running time is expected to be polynomial in $p$, $d_x$, $d_y$, $r$, and $N$.

In order to get started, we make Assumption~\ref{assumption:nicelift} about the existence of a nice lift and integral bases with nice properties.  In practice, we can often find a nice lift and automatically compute the required integral bases in polynomial time (Remark~\ref{remark:overnumfield}).  We also must assume \cite[Assumption 2]{TuitmanII} concerning the orders of poles of the functions making up the integral bases to make use of Tuitman's complexity analysis.

We say a running time is $\widetilde{O}(f(n))$ if the running time is $O(f(n) \log(f(n))^m)$ for some integer $m$.

\subsection{Computing a Crystalline Basis}  As discussed in Section~\ref{sec:computing cohomology}, we first attempt to compute functions, differentials, and uniformizers for use in Theorem~\ref{theorem:computingderham}.  We do so by working over a number field where we can make use of existing algorithms for computing Riemann-Roch spaces and spaces of differentials, and must assume that these give results which work over $\Zq$ as well (Assumption~\ref{assumption:funcdif}).  These run in polynomial time (Remark~\ref{remark:RRpoltime}). As in Theorem~\ref{theorem:computingderham}, we then compute a basis $\Bcrys$ for $\Hcrys^1(X) \simeq \HdR^1(\Xlift)$ in polynomial time.

\subsection{Computing a Rigid Basis}  The time for computing an approximation to $\Brig$ is in $\widetilde{O}(\log(p) d_x^{5} d_y^3 r \Nint)$: see \cite[\S4.1]{TuitmanII}, and simplifying using a crude approximation on the exponent $\theta$ for matrix multiplication.  

\begin{remark}
In practice, we actually compute $\Brig$ exactly (over a number field) to correctly analyze the precision.  The same reasoning shows we need $O(d_x^{2 \theta} d_y^\theta) \subset O(d_x^5 d_y^3)$ operations in that number field.
\end{remark}

\subsection{Computing a Matrix for Frobenius}  We use Tuitman's algorithm to compute $[F]_\rig$.  Following \cite[\S4.2,4.3]{TuitmanII}, we see this may be done in time
\[
\widetilde{O}(p d_x^4 d_y^2 (\Nint + d_x) r \Nint + d_x^5 d_y^3 r \Nint).
\]
The internal precision $\Nint$ must satisfy the condition in Theorem~\ref{thm:precision}: by Remark~\ref{remark:asymptotics} we may choose $\Nint \in \widetilde{O}(N)$.

\begin{remark}
To analyze the actual implementation (correcting for the disagreement between $\Lrig$ and $\Lrig'$), we would need to control $\gamma_1,\gamma_2,\gamma_3$ (see Remark~\ref{remark:don'tanalyzegamma}).
\end{remark}

\subsection{Change Basis}  Expressing $\Bcrys$ in terms of $\Brig$ to change basis is linear algebra over $\Qq$ (although the implementation must work over a number field).  The crystalline basis vector will already lie in $E_0 \cap E_\infty$, so it suffices to apply the map $\pr$ from \eqref{eq:pr}, i.e. multiply by the matrix representing this map.  As analyzed \cite[\S4.3]{TuitmanII}, this can be done in time $\widetilde{O}(\log(p) d_x^5 d_y^3 r \Nint)$.

\subsection{Miscellaneous Steps}
Computing $V = p F^{-1}$ is another linear algebra computation which can be performed in polynomial time.
There are also a few steps intentionally left out of Tuitman's analysis: see \cite[Remark 4.12]{TuitmanII} for a discussion.  In practice their running time is insignificant.  

\begin{remark} \label{remark:runningtime}
In practice, the two most expensive steps are the function field algorithms to compute the functions, differentials, and local expansions needed in Theorem~\ref{theorem:computingderham} and the piece of Tuitman's algorithm computing the matrix for Frobenius on rigid cohomology.  Approximating the crystalline cohomology of the Artin-Schreier curve in Example~\ref{ex:precision} with $7$-adic precision $10$, these two steps take $10.6$ seconds and $48.9$ seconds respectively, while the entire computation takes $73.6$ seconds.\footnote{Running on an Intel(R) Core(TM) i9-10900X CPU @ 3.70GHz.}
\end{remark}

\providecommand{\bysame}{\leavevmode\hbox to3em{\hrulefill}\thinspace}
\providecommand{\MR}{\relax\ifhmode\unskip\space\fi MR }
\providecommand{\MRhref}[2]{%
  \href{http://www.ams.org/mathscinet-getitem?mr=#1}{#2}
}
\providecommand{\href}[2]{#2}

\end{document}